\documentclass[times]{amsart}

\usepackage{color,xcolor}
\usepackage{enumerate,amssymb,amsfonts,amsmath,amstext,amsthm,amscd,etoolbox,mathrsfs,verbatim}
\usepackage{graphicx,graphics}

\newcommand{\D}{\mathbb{D}}
\newcommand{\R}{\mathbb{R}}
\newcommand{\N}{\mathbb{N}}

\newcommand\interior{\operatorname{int}}

\newtheorem{theorem}{Theorem}[section]

\newtheorem{example}[theorem]{Example}

\newtheorem{lemma}[theorem]{Lemma}

\newtheorem{corollary}[theorem]{Corollary}

\theoremstyle{definition}
\newtheorem{definition}[theorem]{Definition}
\newtheorem{remark}[theorem]{Remark}

\begin{document}

\title{Genus $2$ Cantor Sets}

\author{Alastair Fletcher and Daniel Stoertz}


\begin{abstract}
We construct a geometrically self-similar Cantor set $X$ of genus $2$ in $\R^3$. This construction is the first for which the local genus is shown to be $2$ at every point of $X$. As an application, we construct, also for the first time, a uniformly quasiregular mapping $f:\R^3 \to \R^3$ for which the Julia set $J(f)$ is a genus $2$ Cantor set.
\end{abstract}

\maketitle

\section{Introduction}
\subsection{Cantor sets embedded in Euclidean space}

A Cantor set is a totally disconnected, perfect, compact metric space. This is a natural generalization of the standard ternary Cantor set $\mathcal{C}$ contained in a line. Viewed purely as metric spaces, all Cantor sets are homeomorphic to each other. The situation gets more complex, however, once Cantor sets are embedded into $\R^n$. Examples of Cantor sets that can be embedded into $\R^n$ in inequivalent ways to $\mathcal{C}$ embedded in an axis in $\R^n$ go back to Antoine \cite{An}. These constructions are called Antoine necklaces. There is an extensive literature concerning the fascinating and often counter-intuitive properties exhibited by these embeddings. We mention as just two such examples work of Blankinship \cite{Blankinship} and DeGryse and Osborne \cite{DO}.

It is well-known that all Cantor sets embedded into $\R$ or $\R^2$ are, respectively, equivalent to each other. In two dimensions, this property can be phrased by saying there is a disk system which generates the Cantor set. This means that there is a sequence $(D_i)_{i=1}^{\infty}$ of nested sets in the plane, such that each $D_i$ consists of finitely many closed topological disks, and the intersection of all the $D_i$ is precisely the Cantor set. We refer to Moise \cite{Moise} for a proof of this result.

This viewpoint generalizes to three dimensions profitably. Here, every Cantor set embedded in $\R^3$ has a defining sequence given by $(M_i)_{i=1}^{\infty}$, where each $M_i$ is a finite collection of closed handlebodies, $M_{i+1}$ is contained in the interior of $M_i$, and the infinite intersection of the $M_i$ yields the Cantor set under consideration. 

A result of Bing \cite{Bing} shows that a Cantor set $X \subset \R^3$ has a defining sequence consisting of topological balls if and only if there is an ambient homeomorphism $f:\R^3 \to \R^3$ with $f(X) = \mathcal{C}$, where we view the standard ternary Cantor set $\mathcal{C}$ as being contained in one of the coordinate axes. Cantor sets with this property (in any dimension) are called {\it tame}. Consequently, every Cantor set in dimension one and two is tame. Cantor sets which are not tame, such as Antoine's necklace, are called {\it wild}.

\subsection{The genus of a Cantor set}

To further classify Cantor sets embedded in $\R^3$, \v{Z}eljko \cite{Z1} introduced a homeomorphic invariant called the {\it genus} of a Cantor set. Informally, this non-negative integer gives the smallest genus of handlebodies that are required by any defining sequence of the Cantor set. For example, if the Cantor set is tame, there is a defining sequence consisting of topological balls, each of which has genus zero, and hence we say the Cantor set has genus zero. Of course, any defining sequence can be modified by adding extraneous handles, and so the notion of genus is meant to remove such additions.

Antoine's necklace is usually defined via a defining sequence of genus one handlebodies, and a little extra work shows that any defining sequence for Antoine's necklace must consist of handlebodies of genus at least one. Hence Antoine's necklace has genus one. In conjunction with the aforementioned result of Bing, this shows that Antoine's necklace is wild.

There appear to be very few constructions of Cantor sets in $\R^3$ which are proved to be of genus at least $2$ in the literature. \v{Z}eljko's paper \cite{Z1} does construct Cantor sets of every genus, including genus infinity, but is somewhat special in that there is one point of the Cantor set where the higher genus behavior happens, and elsewhere the Cantor set looks like an Antoine's necklace. A refinement of this construction, combined with a construction of Skora \cite{Skora}, is given in \cite{GRZ}, although it is not shown that the genus is at least two.

The only other construction that the authors' are aware of belongs to Babich \cite{Babich}. The terminology of {\it genus} was not yet available to her, but she showed that a certain type of Cantor set called {\it scrawny} could not have a defining sequence of genus one handlebodies. Consequently, her example is of genus two. We refer to \cite{Babich} for more details, but roughly speaking, a scrawny Cantor set $X$ is one for which every embedding of $S^1$ into $\R^3 \setminus X$ of small enough diameter bounds a topological disk which intersects $X$ in finitely many points.

A refinement of the idea of the genus of a Cantor set is the notion of {\it local genus}. If $x\in X$, we consider all the possible defining sequences of $X$ and minimize the genus of the handlebodies in all these defining sequences containing $x$. Clearly, for a tame Cantor set, the local genus is $0$ everywhere. Moreover, for Antoine's necklace the local genus is $1$ everywhere. In the example of \v{Z}eljko mentioned above, the local genus is $1$ everywhere except at one point. The local genus for Babich's example has not been computed.

\subsection{The main result}

The main purpose of the current paper is to give a construction of a genus $2$ Cantor set more in the spirit of Antoine's necklace.
Antoine's necklace can be realized as the attractor set of an iterated function system (IFS) generated by conformal contractions $\varphi_1,\ldots, \varphi_m$ of $\R^3$. We say that any set which can be realized as an attractor set of a conformal IFS is {\it geometrically self-similar}. Our main theorem then reads as follows.

\begin{theorem}
\label{thm:cantorset}
There exist geometrically self-similar Cantor sets $X \subset \R^3$ of genus $2$ with local genus $2$ at every point of $X$.
\end{theorem}

We emphasize that this is the first construction where it is proved that the local genus is larger than $1$ everywhere. Defining sequences give an easy upper bound for the genus of a Cantor set, but finding a lower bound is usually a much harder task. Babich introduced in \cite{Babich} the method of slicing disks, which was also employed in \cite{GRZ}, to give a lower bound for the local genus. In place of this method, in the proof of Theorem \ref{thm:cantorset} we will directly show that genus one solids cannot appear in a non-trivial way in any defining sequence for $X$. As far as we are aware, this technique has not been used in higher genus constructions before.

To compare our construction with constructions in the literature, it is clear that \v{Z}eljko's example is not geometrically self-similar, but Babich's construction could be modified to obtain geometric self-similarity. However, Babich's construction is scrawny and it is not hard to show, although we will not need this, that our construction is not scrawny (in fact, if our construction were scrawny then it would immediately imply genus $2$ and would have made our lives easier). Since Antoine's necklace is not scrawny, this is one of the senses in which our construction is more in the spirit of Antoine's necklace.

We remark that our construction can be achieved as the attractor set of an IFS generated by $32$ contractions. For genus $1$ constructions with round tori, \v{Z}eljko \cite{Z2} has shown that $20$ is the minimal number. Consequently, while our number of $32$ is not sharp for geometrically self-similar genus $2$ Cantor sets, it cannot be far away. We leave it as an open question to determine the minimal number in the genus $2$ setting.

\subsection{Application to dynamics}

Our motivation for constructing these Cantor sets in $\R^3$ comes from dynamics. It is well-known that Cantor sets can arise as Julia sets of rational maps. For example, the Julia set of the quadratic map $f_c(z) = z^2+c$ is a Cantor set in the plane if and only if $c$ does not lie in the Mandelbrot set in parameter space.

The natural generalization of complex dynamics to $\R^3$ (and higher dimensions, although we restrict to dimension $3$ in this paper) is given by the iteration theory of uniformly quasiregular mappings. We postpone the definition of these mappings until the next section, but roughly speaking they are mappings with a uniform bound on the distortion of all the iterates. It follows that the quasiregular version of Montel's Theorem can be applied to such mappings and hence the definitions of the Julia set and Fatou set pass through almost word for word.

The very first paper to study uniformly quasiregular mappings by Iwaniec and Martin \cite{IM} actually constructed one where the Julia set is a tame Cantor set in $\R^3$. The conformal trap construction  by Martin \cite{Martin} also gives a class of uniformly quasiregular maps for which the Julia set is a tame Cantor set. The first construction of a uniformly quasiregular map with the Julia set being a wild Cantor set was given by the first author and Wu in \cite{FW}, where in fact the Julia set is an Antoine's necklace and so has genus $1$.

Using the construction from Theorem \ref{thm:cantorset}, we are able to generalize this construction and give, for the first time, a uniformly quasiregular map whose Julia set is a Cantor set of genus $2$.

\begin{theorem}\label{thm:uqrmap}
	Let $m$ be sufficiently large square that is also a multiple of $16$. Then there exists a Cantor set $X \subset \R^3$ of genus $2$ and a uniformly quasiregular map $f:\R^3 \to \R^3$ of degree $m$ whose Julia set $J(f)$ is $X$.
\end{theorem}

The construction in Theorem \ref{thm:uqrmap} has some interesting consequences for the dynamics of any Poincar\'e linearizer $L$ of a repelling periodic point of the uqr map $f$. The fast escaping set of $L$ must be a spider's web, and contain arbitrarily large genus $2$ surfaces. To keep the focus of this paper on topology, we will not discuss this topic further here, and refer the interested reader to \cite{FS} and the references therein.

The paper is organized as follows. In Section \ref{sec:prelim}, we recall relevant facts about Cantor sets embedded in Euclidean space, the definition of quasiregular mappings, and state results we will need for our constructions. In Section \ref{sec:cantorcon}, we explicitly construct a geometrically self-similar Cantor set $X$. The construction of the uqr map for Theorem \ref{thm:uqrmap} is contained in Section \ref{sec:uqrmap}. In Section \ref{sec:genus}, we prove that the genus of $X$ is 2.

{\bf Acknowledgments:} The authors would like to thank Jang-Mei Wu, whose insight that the first attempt at Theorem \ref{thm:cantorset} only gave a genus $1$ Cantor set put this work on the right track; Vyron Vellis for interesting conversations on the topic of this paper; and the anonymous referee for a number of comments which helped improve the readability of this paper. The results in this paper formed part of the Ph.D. dissertation of the second named author.

\section{Preliminaries}\label{sec:prelim}
We denote by $B(x,r)$ the Euclidean ball at $x \in \R^n$ of radius $r>0$ and by $S(x,r)$ the boundary of $B(x,r)$. Denote by $d(\cdot, \cdot)$ the Euclidean metric on $\R^n$.

\subsection{Cantor sets}\label{sec:cantor}
Recall that a Cantor set is any metric space homeomorphic to the usual Cantor ternary set. Two Cantor sets $E_1, E_2 \subset \R^n$ are equivalently embedded (or ambiently homeomorphic) if there exists a homeomorphism $\psi:\R^n \to \R^n$ such that $\psi(E_1) = E_2$. If the Cantor set $E\subset \R^n$ is equivalently embedded to a usual Cantor ternary set in a line, then $E$ is called \emph{tame}. A Cantor set which is not tame is called \emph{wild}. The first example of a wild Cantor set was Antoine's necklace \cite{An}. Its construction is well-known, but we recall it here for ease of future discussion.

\begin{example}\label{ex:antoine}
	Let $A_0 \subset \R^3$ be a solid torus and let $m\geq 4$ be a positive even integer. Choose mutually distinct solid tori $A_{1,1}, \ldots, A_{1,m}$ contained in the interior of $A_0$ such that $A_{1,i}$ and $A_{1,j}$ are linked if and only if $|i-j|\equiv \pm1 (\text{mod } m)$ and, when linked, they form a Hopf link. Fix homeomorphisms $\varphi_j:A_0 \to A_{1,j}$ for $j \in \{1,\ldots, m\}$, and define
		\[A_1 = \bigcup_{j=1}^m A_{1,j} = \bigcup_{j=1}^m \varphi_j(A_0). \]
	Then inductively define
		\[A_{i+1} = \bigcup_{j=1}^m \varphi_j(A_i),\]
	for $i\geq 1$. See \cite[Section 3.H]{Rolfsen} for a detailed illustration of the first few stages of this process. An Antoine's necklace is defined as
		\[A = \bigcap_{i=1}^\infty A_i. \]
	If $c_i$ is the maximum diameter of the tori in $A_i$, then we require $c_i \to 0$ as $i\to \infty$ in order to ensure that $A$ is a Cantor set. Furthermore, for $m$ sufficiently large, it is possible to construct $A$ such that it is geometrically self-similar (see \cite{Z2}).
\end{example}

Other examples of Cantor sets in $\R^n$ are typically defined in terms of a similar construction to above, using an intersection of nested unions of compact $n$-manifolds with boundary. For Cantor sets in $\R^3$, this notion is well summarized by \v{Z}eljko in \cite{Z1}.

\begin{definition}[see \cite{Z1}, Section 2]\label{def:defseq}
	A \emph{defining sequence} for a Cantor set $E\subset \R^3$ is a sequence $(M_i)$ of compact 3-manifolds with boundary such that
	\begin{enumerate}[(i)]
		\item each $M_i$ consists of disjoint polyhedral cubes with handles,
		\item $M_{i+1} \subset M_i$ for each $i$, and
		
		\item $E = \bigcap_i M_i$.
	\end{enumerate}
	We denote the set of all defining sequences for $E$ by $\mathcal{D}(E)$.
\end{definition}

Using different terminology, Armentrout proved in \cite{Ar} that every Cantor set in $\R^3$ has a defining sequence. Through defining sequences, \v{Z}eljko establishes a useful invariant for Cantor sets. Toward this, for a cube with handles $M$, denote by $g(M)$ the number of handles of $M$. For a disjoint union of cubes with handles $M = \sqcup_{\lambda \in \Lambda} M_\lambda$, define $g(M) = \sup\{g(M_\lambda) : \lambda \in \Lambda\}$.

\begin{definition}[see \cite{Z1}, p.350]\label{def:genus}
	Let $(M_i)$ be a defining sequence for the Cantor set $E \subset \R^3$. Define
		\[g(E;(M_i)) = \sup\{g(M_i) : i \geq 0 \}.\]
	Then define the \emph{genus} of the Cantor set $E$ as
		\[g(E) = \inf\{g(E;(M_i)) : (M_i) \in \mathcal{D}(E)\}. \]
	Now let $x \in E$. Denote by $M_i^x$ the component of $M_i$ containing $x$. Similar to above, define
		\[g_x(E;(M_i)) = \sup\{g(M_i^x) : i \geq 0 \}.\]
	Then define the \emph{local genus of $E$ at the point $x$} as
		\[g_x(E) = \inf\{g_x(E;(M_i)) : (M_i) \in \mathcal{D}(E)\}. \]
\end{definition}

In the same paper, \v{Z}eljko shows that Cantor sets of every genus exist. Also note that if $g(E_1) \neq g(E_2)$, then $E_1$ and $E_2$ are not ambiently homeomorphic.

\subsection{Quasiregular maps}\label{sec:qr}
A mapping $f:E \to \R^n$ defined on a domain $E \subset \R^n$ is called \emph{quasiregular} if $f$ belongs to the Sobolev space $W_{n,loc}^1(E)$ and there exists $K \in [1,\infty)$ such that
	\begin{equation} \label{eq:qr}
		|f'(x)|^n \leq KJ_f(x)
	\end{equation}
almost everywhere in $E$. Here $J_f(x)$ denotes the Jacobian determinant of $f$ at $x\in E$. Informally, a quasiregular mapping extends the behavior of holomorphic mappings in the plane, sending infinitesimal spheres to infinitesimal ellipsoids of uniformly bounded eccentricity. See Rickman's monograph \cite{Rickman} for more details on quasiregular mappings.

A mapping $f:E\to \R^n$ defined on a domain $E\subset \R^n$ is said to be of \emph{bounded length distortion} (BLD) if $f$ is sense-preserving, discrete, open and satisfies
	\[\ell(\gamma)/L \leq \ell(f\circ \gamma) \leq L\ell(\gamma) \]
for some $L \geq 1$ and all paths $\gamma$ in $E$, where $\ell(\cdot)$ denotes the length of a path. BLD mappings were introduced by Martio and V\"ais\"al\"a \cite{MV}. They form a strict subclass of quasiregular mappings.

\subsection{Uqr mappings}\label{sec:Uqr}
While the composition of two quasiregular mappings is again quasiregular, the dilatation typically increases. A quasiregular mapping $f$ is called \emph{uniformly quasiregular}, or uqr, if \eqref{eq:qr} holds uniformly in $K$ over all iterates of $f$. It is hence natural to study the dynamics of iterated uqr mappings. If $f:\R^n \to \R^n$ is uqr, then $x\in \R^n$ is in the Fatou set $F(f)$ if there is a neighborhood $ U \ni x$ such that $ (f^m|_U)_{m=1}^\infty $ forms a normal family.
The Julia set is the complement of the Fatou set, that is $J(f) = \R^n \setminus F(f)$, see \cite{IM}. The escaping set of a quasiregular mapping is
	\[I(f) = \{x \in \R^n : |f^m(x)| \to \infty \text{ as } m\to\infty \}. \]
The following result is a useful tool for determining the Julia set of certain uqr mappings.

\begin{theorem}[Lemma 5.2, \cite{FN}] \label{thm:escaping}
	Let $f:\R^n \to \R^n$ be uqr. Then $J(f) = \partial I(f)$.
\end{theorem}

There exist higher-dimensional uqr counterparts to complex power mappings, constructed by Mayer \cite{Mayer}.

\begin{theorem}[Theorem 2, \cite{Mayer}] \label{thm:power}
	For every $d \in \N$ with $d>1$, there is a uqr map $g:\overline{\R^3} \to \overline{\R^3}$ of degree $d^2$, with Julia set $J(g) = S(0,1)$ and whose Fatou set consists of $B(0,1)$ and $\overline{\R^3} \setminus \overline{B(0,1)}$.
\end{theorem}

In particular, for any $r>0$,
\begin{equation}\label{eq:power}
	g(B(0,r)) = B(0,r^d).
\end{equation}

\subsection{Extending branched coverings}\label{sec:branched}
We require a generalization of the following result of Berstein and Edmonds \cite{BE} on extending coverings over PL cobordisms.

\begin{theorem}[Theorem 6.2, \cite{BE}] \label{thm:BE}
	Let $W$ be a connected, compact, oriented PL $3$-manifold in some $\R^n$ whose boundary $\partial W$ consists of two components $M_0$ and $M_1$ with the induced orientation. Let $W' = N \setminus (\interior B_0 \cup \interior B_1)$ be an oriented PL $3$-sphere $N$ in $\R^4$ with two disjoint polyhedral $3$-balls removed, and have the induced orientation on its boundary. Suppose that $\phi_i:M_i^2 \to \partial B_i$ is a sense-preserving oriented branched covering of degree $d \geq 3$, for each $i=0,1$. Then there exists a sense-preserving PL branched cover $\phi:W \to W'$ of degree $d$ that extends $\phi_0$ and $\phi_1$.
\end{theorem}

Through the work of Heinonen and Rickman \cite{HR} and Pankka, Rajala, and Wu \cite{PRW}, this theorem is known to be true for degree $d \geq 3$ branched covers $\partial W \to \partial W'$ between boundaries of connected, compact, oriented $3$-manifolds $W$ and $W'$, when $\partial W$ has $p \geq 2$ connected components and $W'$ is a PL $3$-sphere with the interiors of $p$ disjoint closed $3$-balls removed.

\section{Construction of the Cantor set}\label{sec:cantorcon}
To construct a geometrically self-similar Cantor set of genus 2, we first need a defining sequence consisting of similar solid double tori. We will also estimate the minimal number of double tori required in the inductive step of the construction to achieve geometric self-similarity. Throughout this section, we make several geometrically convenient choices that have no topological significance.

\subsection{Square tori}\label{sec:square}
Consider a solid torus in $\R^3$ whose core curve is a square. More specifically, start with a circle of radius $R>0$ in the $x_3 =0$ plane. Circumscribe a square $y$ around the circle such the sides of the square are parallel to the $x_1$- and $x_2$-axes, respectively. Now let $Y$ be the result of thickening $y$ by some value $0<r<R$ with respect to the $\infty$-metric in the $x_3=0$ plane. We then obtain a solid torus $T := Y \times [-r,r]$. The core square has sides of length $2R$ and $T$ consists of beams with square cross-sections with sides of length $2r$.

Allowing two such square tori to overlap at a corner and taking their union, we obtain a 
solid double torus. A cross-section in the $x_3=0$ plane showing certain geometric measurements can be seen in Figure \ref{fig:Square2Torus}. Call the constituent solid tori $X_0^1$ and $X_0^2$, with core square curves $\gamma_0^1$ and $\gamma_0^2$, respectively. Then call the solid double torus $X_0 = X_0^1 \cup X_0^2$ with core curve $\gamma_0 = \gamma_0^1 \cup \gamma_0^2$.
\begin{figure}[h]
	\centering \includegraphics[width = 3.5in]{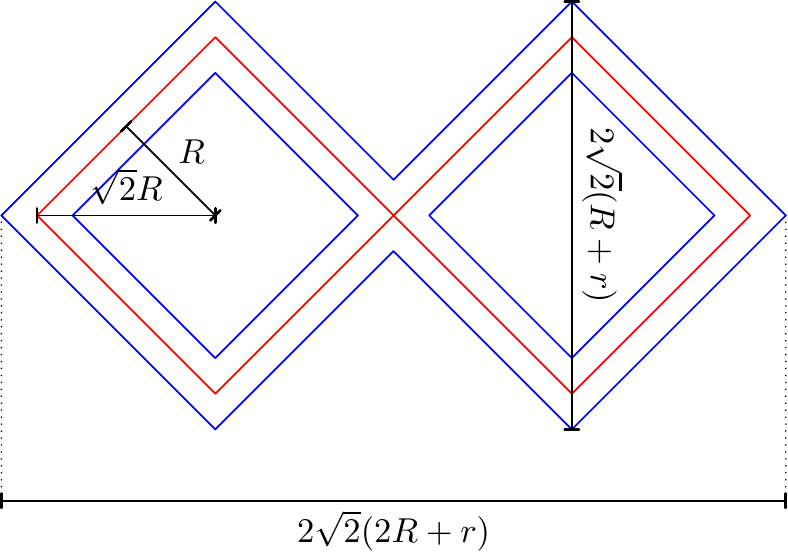}
	\caption{A square double torus.}\label{fig:Square2Torus}
\end{figure} 

For a suitably chosen integer $m$, we now wish to construct solid double tori $X_{1,1}, \ldots, X_{1,m}$ contained in $\interior X_0$ that are linked, are geometrically similar to $X_0$ with common scaling factor $k$, and are arranged along $\gamma_0$ in such a way that the resulting Cantor set has genus 2. Note that $k$ and $m$ depend on each other, and both depend on the value of $r$.

The linking between neighboring double tori is as follows. As with $X_0$, we regard each $X_{1,j}$ as the union of overlapping solid tori $X_{1,j}^1$ and $X_{1,j}^2$, with respective core squares $\gamma_{1,j}^1$ and $\gamma_{1,j}^2$. Then let $\gamma_{1,j} = \gamma_{1,j}^1 \cup \gamma_{1,j}^2$ be the core curve of $X_{1,j}$. We want $\gamma_{1,j}^1$ to form a Hopf link with $\gamma_{1,j-1}^2$, and $\gamma_{1,j}^2$ to form a Hopf link with $\gamma_{1,j+1}^1$, modulo $m$. See Figure \ref{fig:2Link} for an illustration. The angle between subsequent double tori is chosen to be $\pi/2$ for simplicity.
\begin{figure}[h]
	\centering\includegraphics[width = 3.5in]{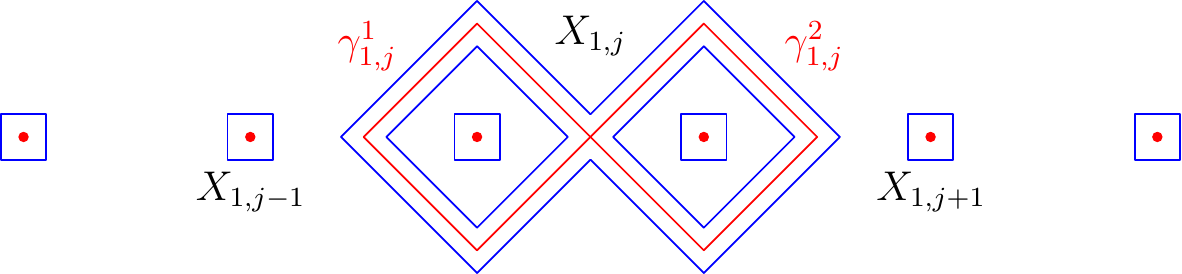}
	\caption{A link between three double tori.}
	\label{fig:2Link}
\end{figure}

Furthermore, the double tori $X_{1,1}, \ldots, X_{1,m}$ are to be arranged along $\gamma_0$ so that the resulting chain is shaped like a figure-eight. To achieve this, we require that there be a four-way linking of double tori at the self-intersection point of $\gamma_0$. See Figure \ref{fig:4WayCore} for an illustration of a four-way linking between the core curves of some double tori.
\begin{figure}[h]
	\centering \includegraphics[width=2in]{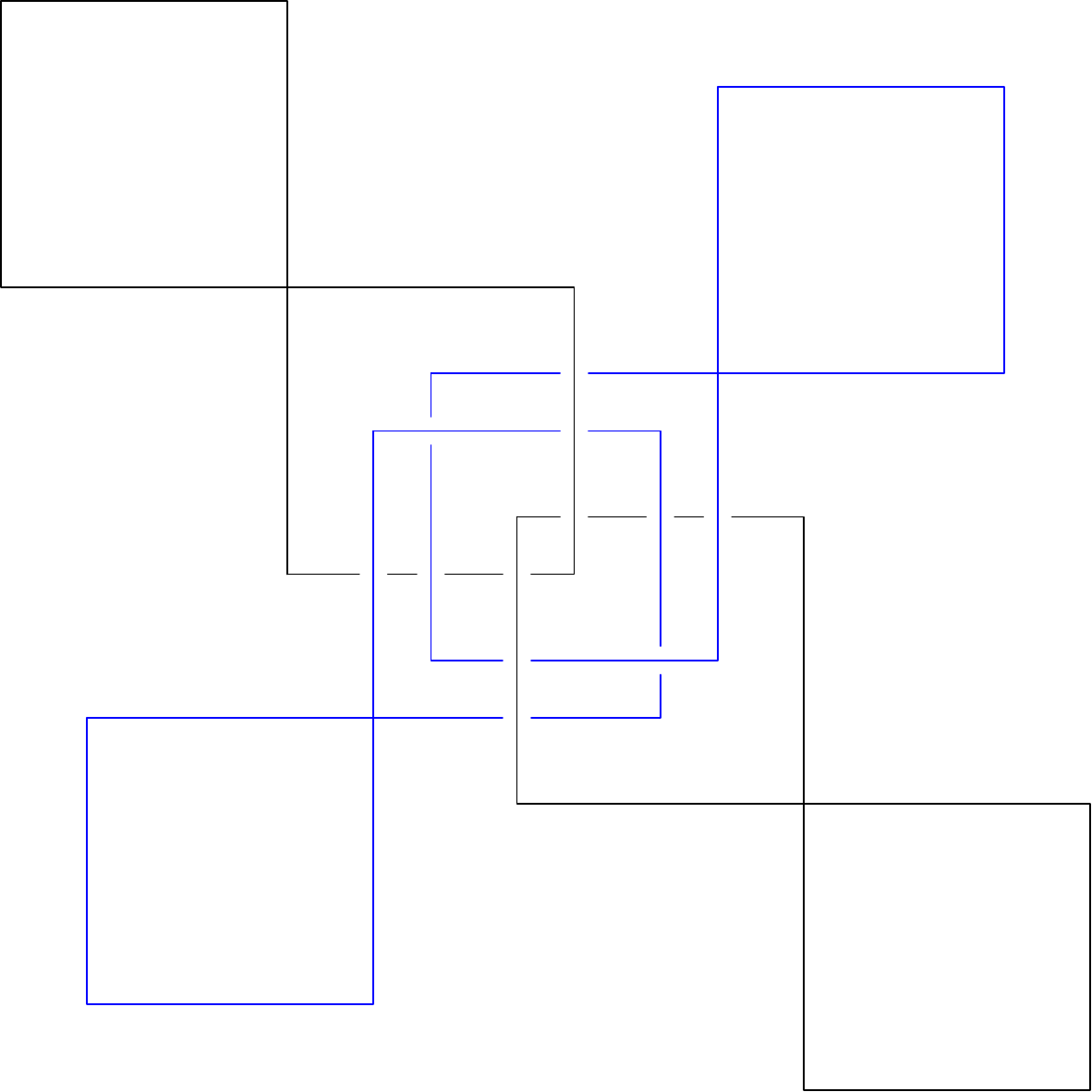}
	\caption{A four-way linking between figure-eight core curves.
	} \label{fig:4WayCore}
\end{figure}
Given a specific choice of position and orientation for the four double tori in question, we will need a bound on the thickness coefficient $r$ to ensure the tori are mutually disjoint.

\subsection{The four-way linking}\label{sec:4way}

It will suffice to consider a four-way linking of solid tori, since our double tori are then obtained by taking the union with a second overlapping torus.

For computational simplicity, place the self-intersection point of $\gamma_0$ at the origin, and orient $\gamma_0$ so that the segments emanating from the self-intersection point follow the $x_1$- and $x_2$-axes, respectively. Choose four square tori, call them $T_1, \ldots, T_4$ with square core curves $\gamma_1, \ldots, \gamma_4$, respectively. Orient the tori so that the intersections of the $\gamma_j$ and the $x_1x_2$-plane happen along the diagonal of the $\gamma_j$. 

Then position the intersection with the $x_1x_2$-plane as in Figure \ref{fig:4WayLink}.
\begin{figure}[h]
	\centering\includegraphics[width = 5in]{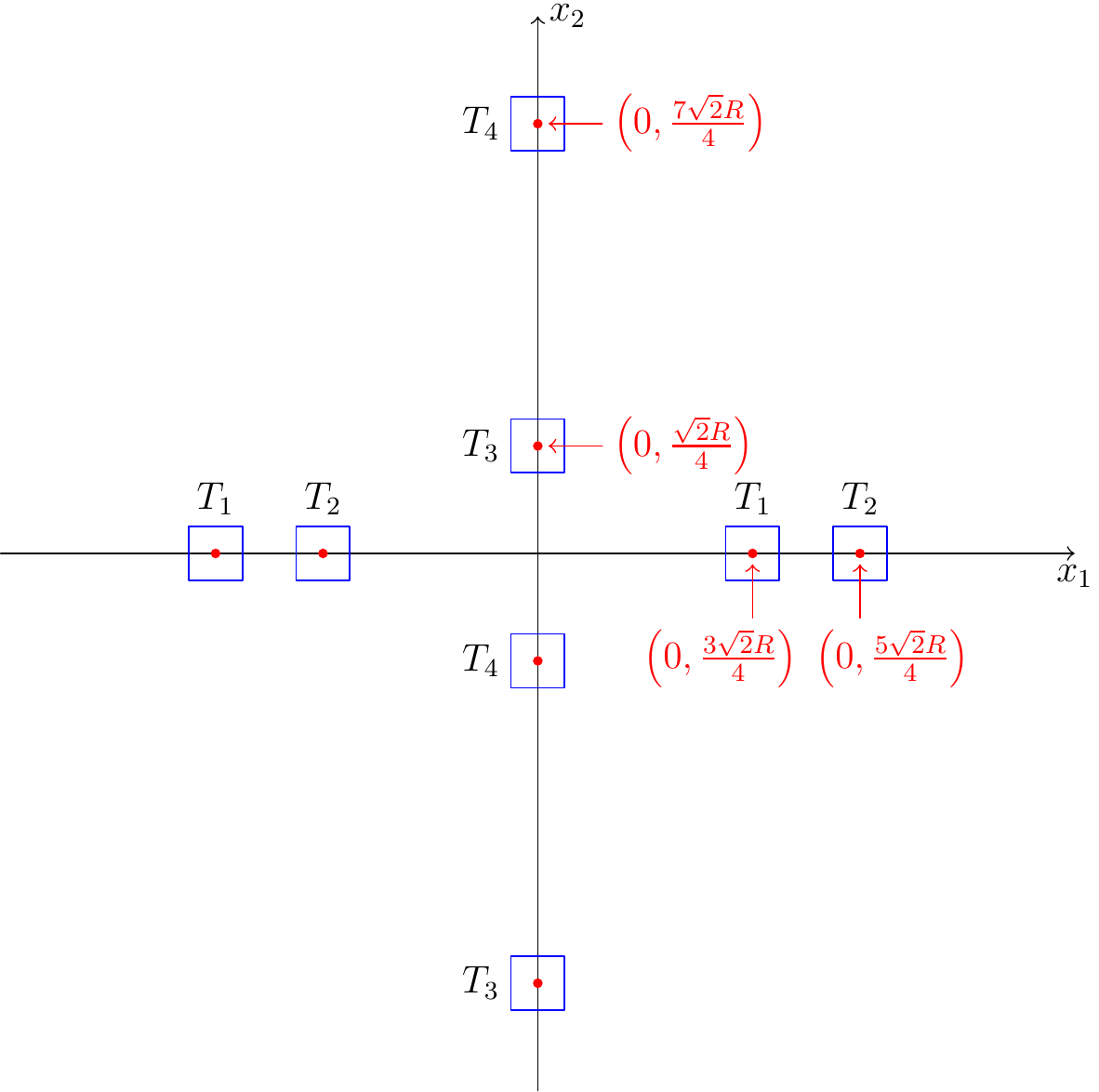}
	\caption{The four-way linking, before rotation.
	}\label{fig:4WayLink}
\end{figure}
The coordinates of the intersection points between the $\gamma_j$ and the $x_1x_2$-plane are chosen to be:\\
\begin{itemize}
	\item $T_1:\, (3\sqrt{2}R/4 ,0,0), (-5\sqrt{2}R/4, 0,0)$
	\item $T_2:\, (5\sqrt{2}R/4 ,0,0), (-3\sqrt{2}R/4, 0,0)$
	\item $T_3:\, (0, 1\sqrt{2}R/4,0), (0, -7\sqrt{2}R/4,0)$
	\item $T_4:\, (0, 7\sqrt{2}R/4,0), (0, -1\sqrt{2}R/4,0)$
\end{itemize}

Finally, rotate the tori as follows:
\begin{align}
T_1& \text{ is rotated about the } x_1\text{-axis by an angle of }-3\pi/8 \label{eq:rot1} \\ 
T_2& \text{ is rotated about the } x_1\text{-axis by an angle of }3\pi/8 \label{eq:rot2}\\
T_3& \text{ is rotated about the } x_2\text{-axis by an angle of }\pi/8 \label{eq:rot3}\\ 
T_4& \text{ is rotated about the } x_2\text{-axis by an angle of }-\pi/8 \label{eq:rot4}
\end{align}
Then the $\gamma_j$ are all disjoint and pairwise form Hopf links. To bound $r$, we estimate the distance between the $\gamma_j$. Thanks to the symmetry of the position and orientation of the tori, it suffices to calculate the distance between only four line segments, call them $L_1, \ldots, L_4$, each in the upper half-space (having $x_3 \geq 0$). After the rotations described in \eqref{eq:rot1} - \eqref{eq:rot4}, the chosen lines have the following vector equations:
\begin{align*}
L_1 &: \begin{pmatrix} 3\sqrt{2}R/4 \\ 0 \\ 0 \end{pmatrix} +t \begin{pmatrix} -1\\\sqrt{2 + \sqrt{2}}/2\\ \sqrt{2 - \sqrt{2}}/2\end{pmatrix} &
L_2 &: \begin{pmatrix} 5\sqrt{2}R/4 \\ 0 \\ 0 \end{pmatrix} +t \begin{pmatrix} -1\\-\sqrt{2 +\sqrt{2}}/2\\ \sqrt{2 -\sqrt{2}}/2\end{pmatrix}\\
L_3 &: \begin{pmatrix} 0\\ \sqrt{2}R/4 \\ 0 \end{pmatrix} +t \begin{pmatrix} -\sqrt{2 -\sqrt{2}}/2\\-1\\ \sqrt{2 +\sqrt{2}}/2\end{pmatrix} &
L_4 &: \begin{pmatrix} 0\\ -\sqrt{2}R/4 \\ 0 \end{pmatrix} +t \begin{pmatrix} \sqrt{2 -\sqrt{2}}/2\\ 1\\\sqrt{2 +\sqrt{2}}/2\end{pmatrix},\\
\end{align*}
for $0 \leq t \leq \sqrt{2}R$. Let $\tau_j$ denote the cylinder around $L_j$ with radius $\sqrt{2}r$. Then we have that
\begin{equation}\label{eq:distance}
	d(T_i, T_j) \geq d(\tau_i, \tau_j) = d(L_i, L_j) - 2\sqrt{2}r.
\end{equation}
The smallest value on the right hand side of \eqref{eq:distance} is obtained for both the pairs $(L_1, L_2)$ and $(L_3, L_4)$. This value equals
	\[\frac{R}{2\sqrt{\frac{5}{2} +\sqrt{2}}} -2\sqrt{2}r. \]
This quantity must be greater than $0$, yielding the following inequality.

\begin{lemma}\label{lemma:rbound}
	Let $\gamma_1, \ldots, \gamma_4$ be squares with side length $2R$ and oriented as above. Suppose that $r>0$ satisfies
		\[r < \frac{R}{4\sqrt{5 +2\sqrt{2}}}, \]
	and let $T_1, \ldots, T_4$ be the square tori obtained by thickening $\gamma_1, \ldots, \gamma_4$ by $r$ as in Section \ref{sec:square}. Then $T_1, \ldots, T_4$ are mutually disjoint and any pair of the $\gamma_1, \ldots, \gamma_4$ forms a Hopf link with each other.
\end{lemma}

\subsection{Bounding $m$ and constructing the chain}\label{sec:mbound}

We now construct solid double tori $X_{1,1}, \ldots, X_{1,m}$ along the curve $\gamma_0$ with a four-way linking at the central point of $\gamma_0$. As above, denote by $\gamma_{1,j}$ the core curve of $X_{1,j}$ similar to $\gamma_0$ for $X_0$. Then $X_{1,j}$ has size coefficient $kR$ and thickness coefficient $kr$. Recall that $k$ and $m$ depend on each other. A preliminary bound for $k$, which will be overridden later, is determined so that the length of each $X_{1,j}$ fits inside half the thickness of $X_0$. In other words,
\begin{equation}\label{eq:kbound}
	2\sqrt{2}k(2R +r) <r.
\end{equation}
This ensures that, no matter the orientation of the $X_{1,j}$, as long as $\gamma_{1,j}$ touches $\gamma_0$, we have that $X_{1,j}$ is contained in $\interior X_0$. Henceforth assume that $k$ satisfies \eqref{eq:kbound}.

For simplicity, position $X_0$ such that its central point is at the origin, its length runs along the $x_1$-axis, and its width is parallel to the $x_2$-axis. We can then regard $\gamma_0$ as split into eight line segments, two in each quadrant of the $x_1x_2$-plane. Let $\gamma_0'$ be the segment in the first quadrant emanating from the origin. Note that the length of $\gamma_0'$ is $2R$.

Choose a point $p_1$ on $\gamma_0'$ to be $\sqrt{2}kR/4$ units from the origin. Now let $n$ be a sufficiently large even number, and position points $p_2, \ldots, p_n$ sequentially along $\gamma_0'$ such that the Euclidean distance $d(p_i, p_{i+1}) = 3\sqrt{2}kR$ for $i\in \{1, \ldots, n-1\}$. Note that then $p_n$ is $3\sqrt{2}kR -\sqrt{2}kR/4$ units short of the terminal point of $\gamma_0'$. Finally, let $p_{n+1}$ be the point $\sqrt{2}kR/4$ units past the terminal point of $\gamma_0'$, still in line with $p_1, \ldots, p_n$.

These points will be used as anchors for the double tori $X_{1,1}, \ldots, X_{1,n}$. For ease of discussion, we define orientation vectors as follows. For $j\in\{1. \ldots, m \}$, let $\mathbf{v}_{j,1}$ and $\mathbf{v}_{j,2}$ be the unit vectors in the direction of the length and width of $X_{1,j}$ respectively.

Note that the distance between $p_i$ and $p_{i+1}$ is the same as the distance from a lengthwise corner of $\gamma_{1,i}$ to the midpoint of the opposing hole of $X_{1,i}$. So, for each $i$, position $X_{1,i}$ such that $p_i$ lies at the center of one hole of $X_{1,i}$, and $p_{i+1}$ lies at the opposite terminal corner of $\gamma_{1,i}$ (see Figure \ref{fig:points}).
\begin{figure}[h]
	\centering\includegraphics[width = 3in]{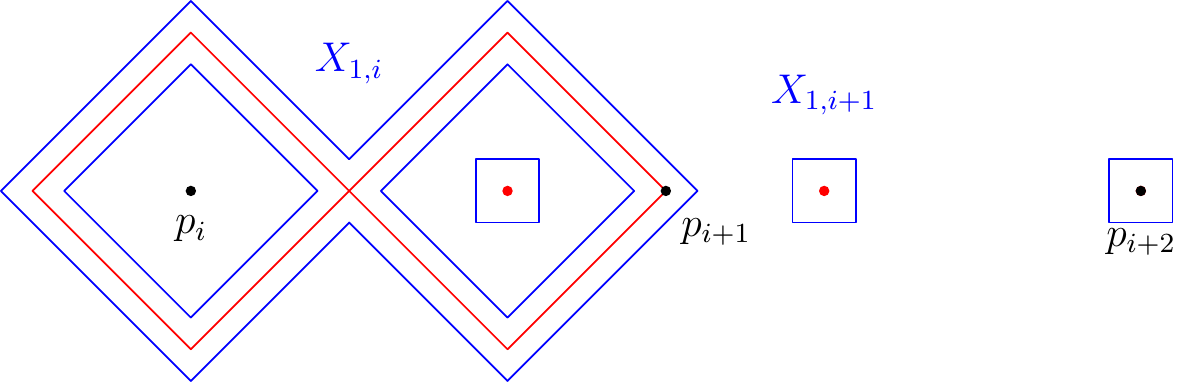}
	\caption{$X_{1,i}$ as determined by $p_i$ and $p_{i+1}$.} \label{fig:points}
\end{figure}
This means that $\mathbf{v}_{i,1} = \frac{\sqrt{2}}{2} \left\langle 1,1,0 \right\rangle$ for each $i$. To ensure the linking of sequential tori, if $i$ is an odd integer between $1$ and $n$, let
\[\mathbf{v}_{i,2} = \left\langle \frac{\sqrt{4 +2\sqrt{2}}}{4},\, \frac{-\sqrt{4 +2\sqrt{2}}}{4},\, \frac{\sqrt{2- \sqrt{2}}}{2} \right\rangle,\]
which has an angle of $3\pi/8$ with respect to the $x_3$-axis, and let
\[\mathbf{v}_{i+1,2} = \left\langle -\frac{\sqrt{4-2\sqrt{2}}}{4},\, \frac{\sqrt{4-2\sqrt{2}}}{4},\, \frac{\sqrt{2 +\sqrt{2}}}{2} \right\rangle,\]
which has an angle of $\pi/8$ with respect to the $x_3$-axis. This way, the widths of the double tori $X_{1,i}$ and $X_{1, i+1}$ are perpendicular to each other for each $i \in \{ 1, \ldots, n\}$, guaranteeing linking with the given spacing.

Let $\rho_1$ denote a clockwise (with respect to the $x_1x_2$-plane) rotation by the angle $\pi/2$ around the vertical line through the point $(2\sqrt{2}R, 0, 0)$ (the center of the hole of $X_0$ in the $x_1\geq 0$ half-space). Then for $i=n+1, \ldots, 2n$, we have that $\mathbf{v}_{i,1} = \frac{\sqrt{2}}{2}\left\langle 1, -1, 0\right\rangle$ and
\[\mathbf{v}_{i,2} =
\begin{cases}
\left\langle \frac{-\sqrt{4 +2\sqrt{2}}}{4},\, \frac{-\sqrt{4+2\sqrt{2}}}{4},\, \frac{\sqrt{2- \sqrt{2}}}{2} \right\rangle & \text{if } i \text{ is odd,}\\
\left\langle \frac{\sqrt{4-2\sqrt{2}}}{4},\, \frac{\sqrt{4-2\sqrt{2}}}{4},\, \frac{\sqrt{2 +\sqrt{2}}}{2} \right\rangle &\text{if } i \text{ is even.}
\end{cases} \]
Note that $X_{1,n}$ and $X_{1, n+1}$ are then linked in the same manner as the tori $T_3$ and $T_2$ from Section \ref{sec:4way}.

Now apply $(\rho_1)^2$, that is, for $i = 1,\ldots,2n$ define $X_{1, i+2n} = (\rho_1)^2(X_{1,i})$. Similarly to the previous rotation, $X_{1,2n}$ and $X_{1, 2n+1}$ are linked, again in the same manner as $T_3$ and $T_2$ from Section \ref{sec:4way}. Additionally, $X_{1,1}$ and $X_{1,4n}$ are now linked, again in a way corresponding to $T_2$ and $T_3$ (in that order).

Finally, let $\rho_2$ denote a rotation by the angle $\pi$ around the $x_3$-axis. For $i = 1,\ldots, 4n$, define $X_{1,i+ 4n} = \rho_2(X_{1,i})$. We then have a chain $X_1 = \bigcup_{j=1}^m X_{1,j}$ of $m = 8n$ linked double tori along $\gamma_0$. By the rotational symmetry of the four-way linking pointed out in Section \ref{sec:4way}, $X_1$ has the desired four-way linking at the origin, consisting of the tori $X_{1,1},\, X_{1,4n},\, X_{1,4n+1},$ and $X_{1,8n}$, corresponding to the tori $T_2,\, T_3,\, T_4$, and $T_1$, respectively.

Note that, since $n$ is even, we have that $m =8n$ is a multiple of 16. To find a minimal $m$ for which this construction is possible, recall that each pair of consecutive points from $p_1, \ldots, p_n$ determines a line segment of length $3\sqrt{2}kR$. So the total length of the path along which the double tori $X_{1,1}, \ldots, X_{1,n}$ are arranged is $3\sqrt{2}kRn$. But this is equal to the length of $\gamma_0'$, which is $2R$. This gives us the equation
\[3\sqrt{2}kRn = 2R.\]
Replacing $n$ with $m/8$ and rearranging yields		 
\begin{equation}\label{eq:mbound}
m = \frac{16}{3\sqrt{2}k}. 
\end{equation}
So, if we choose $k$ to satisfy \eqref{eq:kbound} and such that $m$ in \eqref{eq:mbound} is a multiple of 16, then we can construct the chain $X_1 = \bigcup_{j=1}^m X_{1,j}$.

\begin{example}
	Let $R = 1$. Then $r = 0.08$ satisfies Lemma \ref{lemma:rbound}, and any $k \leq 0.013$ satisfies \eqref{eq:kbound}. Under these conditions, the maximum value of $k$ for which $m$ in \eqref{eq:mbound} is a multiple of 16 is $k = \frac{1}{6\sqrt{2}}$, yielding a chain of $m=32$ tori.
\end{example}

\begin{remark}
	The construction of $X_1$ given in this section is not optimal. Relaxing the constraint on $k$ and arranging the $X_{1,j}$ differently could potentially yield a geometrically self-similar Cantor set using fewer than 32 double tori in $X_1$. However, the value $m=32$ is likely the smallest number of double tori possible using this particular construction method. 
\end{remark}

\subsection{The Cantor Set}\label{sec:cansetconstruction}

Let $X_0$ be a genus 2 solid torus as in Section \ref{sec:square}, with size coefficient $R$ and thickness coefficient $r$ satisfying Lemma \ref{lemma:rbound}. Let $k$ satisfy \eqref{eq:kbound}, and accordingly let $m$ be a multiple of 16 satisfying \eqref{eq:mbound}. Let $X_{1,1},\ldots, X_{1,m}$ be genus 2 solid tori arranged as in Section \ref{sec:mbound}, and fix sense-preserving similarities $\phi_j: X_0 \to X_{1,j}$ for $j = 1,\ldots,m$. Define $X_1 = \bigcup_{j=1}^m X_{1,j}$. Now for $n \geq 2$, define
\[X_n = \bigcup_{j=1}^m \phi_j(X_{n-1}). \]
Then the set
\[X = \bigcap_{n=0}^\infty X_n \]
is a Cantor set with genus at most 2. That $X$ is wild can be seen using any of the many well-known proofs of the wildness of Antoine's necklace (see, for example, \cite{Moise}).

\section{A Julia set of genus $2$}\label{sec:uqrmap}

The construction of a uqr map $f$ of polynomial type having the Cantor set $X$ as its Julia set is a modification of the method used by Wu and the first named author in \cite{FW}.

\subsection{A basic covering map}\label{sec:covering}
Recall the coefficients $R$ and $r$ from Section \ref{sec:cantorcon} describing the size and thickness, respectively, of the solid double torus $X_0$. Set $R=1$, and assume $r>0$ satisfies Lemma \ref{lemma:rbound}. Assume also that $X_0$ is positioned and oriented in $\R^3$ as in Section \ref{sec:mbound}. Towards the proof of Theorem \ref{thm:uqrmap}, we construct a BLD degree $m$ branched covering map
\[F: X_0 \setminus \interior\left(\bigcup_{j=1}^m X_{1,j} \right) \to \overline{B(0,4)} \setminus \interior(X_0) \]
satisfying $F|\partial X_{1,j} : \partial X_{1,j} \to \partial X_0 = \phi_j^{-1}$ for the tori $X_{1,1}, \ldots, X_{1,m}$ fixed in Section \ref{sec:mbound}.

Let $\iota_1$ be the involution
\[\iota_1:(x_1, x_2, x_3) \mapsto (-x_1, -x_2, x_3). \]
By construction, the double tori $X_{1,1}, \ldots, X_{1,m}$ are symmetric with respect to $\iota_1$. The quotient $q_1: X_0 \to X_0/\left\langle \iota_1\right\rangle$ is then a degree 2 sense-preserving map satisfying
\begin{itemize}
	\item $q_1(X_0)$ is a solid torus unknotted in $\R^3$,
	\item $q_1(X_{1,j}) = q_1(X_{1, m-j+1})$ is a double torus unknotted in $q_1(X_0)$,
	\item $\bigcup_{j=1}^m q_1(X_{1,j})$ is a chain of $m/2$ linked double tori following a core curve of the torus $q_1(X_0)$.
\end{itemize}
For the sake of convenient geometry, we modify $q_1(X_0)$ in a few ways. First, translate $q_1(X_0)$ so that the center of the hole of the torus is at the origin. Then apply a map that is radial with respect to the $x_3$-axis, making $q_1(X_0)$ round in two senses:
\begin{itemize}
	\item the core curve traced by the chain $\bigcup_{j=1}^m q_1(X_{1,j})$ is a circle in the $x_1x_2$-plane centered at the origin;
	\item every cross section of $q_1(X_0)$ taken perpendicular to the above core curve is a geometric disk.
\end{itemize}
Additionally, we deform $\interior(q_1(X_0))$ so that all the double tori $q_1(X_{1,1}), \ldots, q_1(X_{1,m})$ satisfy $\rho(q_1(X_{1,j})) = \rho(q_1(X_{1,j+2}))$ for $j\in \{1, \ldots, m-2\}$, $\rho(q_1(X_{1, m-1})) = q_1(X_{1,1})$, and $\rho(q_1(X_{1,m})) = q_1(X_{1,2})$, where $\rho$ is the rotation about the $x_3$-axis by an angle $8\pi/m$,
\[\rho(r, \theta, x_3) = (r, \theta + 8\pi/m, x_3). \]
This deformation is made to preserve the fact that all the $q_1(X_{1,j})$ remain geometrically similar to each other. Finally, if necessary, rotate $q_1(X_0)$ around the $x_3$-axis to ensure that $q_1(X_{1,1}) = q_1(X_{1, m/2+1})$ and $q_1(X_{1,m/2}) = q_1(X_{1, m})$ are linked with the $x_1$-axis such that they are symmetric with respect to a rotation about the $x_1$-axis by an angle $\pi$. For the sake of notational simplicity, assume that the map $q_1$ already incorporates all of these modifications.

Let $\omega:\R^3 \to \R^3$ be the degree $m/4$ winding map
\[\omega(r, \theta, x_3) = (r, \theta m/4, x_3). \]
Then $\omega:q_1(X_0) \to q_1(X_0)$ is an unbranched cover that maps all $q_1(X_{1,j})$ with odd indices to $\omega(q_1(X_{1,1}))$ and all $q_1(X_{1,j})$ with even indices to $\omega(q_1(X_{1,2}))$. By construction, $\omega(q_1(X_{1,1}))$ and $\omega(q_1(X_{1,2}))$ are linked inside $q_1(X_0)$ as in Figure \ref{fig:Wrapping},
\begin{figure}[h]
	\centering\includegraphics[width=4.2in]{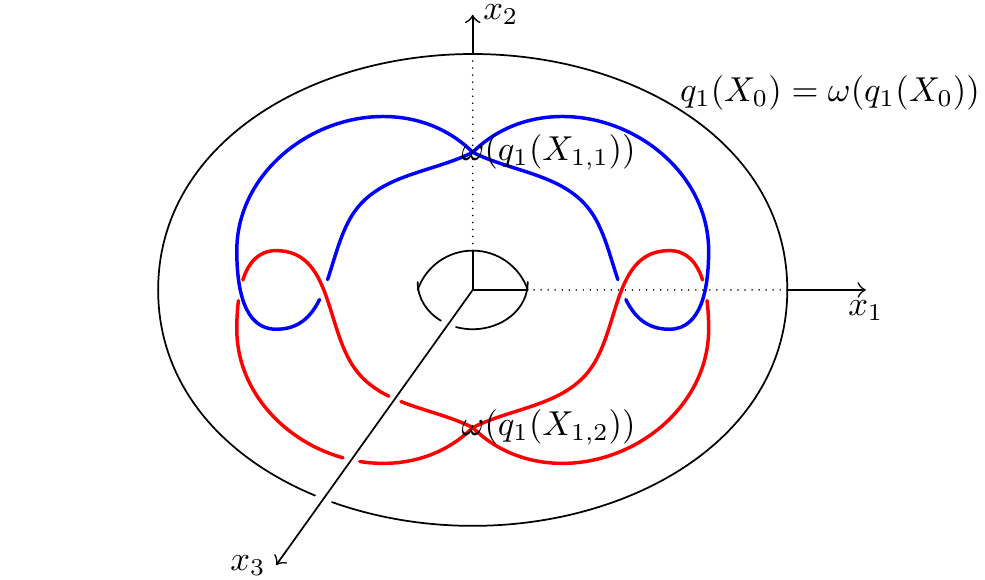}
	\caption{}\label{fig:Wrapping}
\end{figure}
and are symmetric to each other via a rotation about the $x_1$-axis by an angle $\pi$. Let $\iota_2$ be the involution for this rotation, that is
\[\iota_2:(x_1, x_2, x_3) \mapsto (x_1, -x_2, -x_3). \]
The quotient $q_2: q_1(X_0) \to q_1(X_0) / \left\langle\iota_2 \right\rangle$ is then a degree 2 sense preserving map under which $q_2(\omega(q_1(X_{1,1}))) = q_2(\omega(q_1(X_{1,2})))$ is a double torus unknotted in the 3-cell $q_2(q_1(X_0))$. For more details on such constructions, see \cite[p. 294]{Rolfsen}. Assuming $q_2$ incorporates some more translations and deformations, the map $q_2\circ \omega \circ q_1$ is a degree $m$ branched cover from $X_0$ onto $\overline{B(0,4)}$ mapping each $X_{1,j}$ onto $X_0$. To obtain a BLD, and hence quasiregular, cover, we consider a PL version of this map.

Give $X_0$ a $C^1$-triangulation $g:|U| \to X_0$ by a simplicial complex $U$ that respects the involutions $\iota_1$ and $\iota_2$, and has both $g^{-1}(q_1^{-1}(\bigcup X_{1,j}))$ and $g^{-1}(q_1^{-1}(\omega(q_1(X_{1,1})) \cup \omega(q_1(X_{1,2}))))$ as subcomplexes. We then identify $q_1(X_0)$ with a simplicial complex $V$ via $h:|V| \to q_1(X_0)$ such that $q_1 \circ g: U \to q_1(X_0)$ is simplicial. It then follows that
\begin{itemize}
	\item $h^{-1}(\bigcup q_1(X_{1,j}))$ is a subcomplex of $V$,
	\item $h^{-1}(\omega(q_1(X_{1,1})) \cup \omega(q_1(X_{1,2})))$ is a subcomplex of $V$, and
	\item $h$ respects $\iota_2$.
\end{itemize}
Finally, identify $q_2(q_1(X_0))$ with a simplicial complex $W$ via $i:|W| \to q_2(q_1(X_0))$ such that $q_2 \circ q_1 \circ g$ is simplicial. Then $i^{-1}(q_2(\omega(q_1(X_{1,1}))))$ is a subcomplex of $W$.

Refine $|U|$ and $|W|$ if necessary to ensure that $q_2 \circ \omega \circ q_1 \circ \phi_j |X_0$ are simplicial and ambient isotopic. This is possible since $q_2 \circ \omega \circ q_1 \circ \phi_j$ embeds $X_0$ unknottedly into $|W|$. Then there exists a PL map $\eta:|W| \to |W|$ which is identity on $\partial |W|$ so that $\eta \circ i^{-1}|X_0 = q_2 \circ \omega \circ q_1 \circ \phi_j|X_0$. Set $\zeta = \eta\circ i^{-1}$, and then $F := \zeta^{-1} \circ q_2 \circ \omega \circ q_1$ is a BLD degree $m$ branched covering satisfying $F|\partial X_{1,j} = \phi_j^{-1}$.

\subsection{A genus 2 Julia set}\label{sec:juliaset}
Let $m=d^2$ be a sufficiently large square that is a multiple of 16 and let $X$ be the Cantor set from the previous section. Write $B_0 = B(0,4)$, $B_{-1} = B(0,4^d)$, and write $\R^3$ as two disjoint unions, one for the domain and one for the codomain of $f$, as follows:
\[\R^3 = X_1 \cup (X_0 \setminus X_1) \cup (B_0 \setminus X_0) \cup (\R^3 \setminus B_0) \]
and
\[\R^3 = X_0 \setminus (B_0 \setminus X_0) \cup (B_{-1} \setminus B_0) \cup (\R^3 \setminus B_{-1}). \]
The uqr map $f$ is then defined as follows:
\begin{enumerate}
	\item Set $f:\overline{X_0 \setminus X_1} \to \overline{B_0 \setminus X_0}$ to be the degree $m$ branched covering map $F$ from the previous section.
	
	\item Extend $f$ into $X_1$ by setting $f|{X_1}$ to be $\phi_j^{-1}:X_{1,j} \to X_0$ for each $j \in \{1, \ldots, m\}$.
	
	\item Define $f:\R^3 \setminus \interior(B_0) \to \R^3 \setminus \interior(B_{-1})$ to be the restriction of the uqr map $g$ of degree $m$ from Theorem \ref{thm:power}. By definition of the map $g$, it maps $S(0,4)$ onto $S(0,4^d)$. We further remark that $g$ is orientation-preserving.
	
	\item Since $f|\partial B_0$ is a BLD degree $m$ branched cover onto $\partial B_{-1}$ and $f|\partial X_0$ is also a BLD degree $m$ branched cover onto $\partial B_0$, we can extend the boundary map to a BLD degree $m$ branched cover $f:B_0 \setminus \interior (X_0) \to B_{-1} \setminus \interior(B_0)$ by the Berstein and Edmonds extension theorem, Theorem \ref{thm:BE}.
\end{enumerate}
Then $f:\R^3 \to \R^3$ is indeed a quasiregular map. We now prove most of Theorem \ref{thm:uqrmap} with the following two lemmas.
\begin{lemma}\label{lem:uqr1}
	The map $f$ is a uniformly quasiregular mapping of polynomial type.
\end{lemma}
\begin{proof}
	Let $x \in \R^3$. If the orbit of $x$ under $f$ always remains in $X_1$, then, since $f|X_1$ is a conformal similarity, the dilatation of $f^n$ at $x$ will always equal 1.

	Suppose then that the orbit of $x$ leaves $X_1$. Then, after iterating through at most finitely many conformal maps and at most two quasiregular maps, $f^{n_0}(x_0) \in \R^3 \setminus \overline{B(0,4)}$ for some $n_0 \in \N$. From this point on, $f$ agrees with the uqr power map $g$ of degree $m$, and hence the dilatation will remain bounded. In summary, the orbit of $x$ consists of finitely many conformal maps, at most two quasiregular maps, and then a uqr map. So the dilatation of $f^n$ remains uniformly bounded on all of $\R^3$ as $n\to \infty$. Hence $f$ is uqr.

	Since $f$ has finite degree $m$, $f$ is of polynomial type.
\end{proof}
\begin{lemma}\label{lem:uqr2}
	The Julia set of $f$ is equal to $X$.
\end{lemma}
\begin{proof}
	Note first that, by construction, $f(X) = X$.

	Now let $x \in \R^3$. If the orbit of $x$ under $f$ at any point leaves $X_0$, then $x \in I(f)$, as it is pushed to infinity by the uqr power map. By construction, if $x$ does not leave $X_0$, then $x \in X$.

	If $x \in X$, then any sufficiently small neighborhood of $x$ will intersect the boundary of $X_n$ for some $n$. Since $\partial X_n \subset I(f)$ by the preceding argument, we conclude that $X= \partial I(f)$. Since $f$ is uqr, we have by Theorem \ref{thm:escaping} that $\partial I(f) = J(f)$. Hence $J(f) = X$.
\end{proof}
To complete the proof of Theorem \ref{thm:uqrmap}, it remains to prove Theorem \ref{thm:cantorset}. That is, it remains to prove that the Cantor set $X$ has genus 2.

\section{Proving the genus is $2$}\label{sec:genus}

To prove that $g(X) =2$, we will prove Theorem \ref{thm:cantorset}, that is, that the local genus $g_x(X) = 2$ for all $x\in X$. Since $g(X) \geq g_x(X)$ (see \cite{Z1}), this together with the defining sequence from Section \ref{sec:cantorcon} will yield the desired result. Note that, since $X$ is wild, we must have that $g(X) \geq 1$ (see \cite{Bing}).

For this section we will need some basic knot theoretical properties of the solid torus $T$. The arguments in this paragraph follow \cite[Section 2]{BSG} and see, for example, \cite[Section 2.E]{Rolfsen} and \cite[Section 4.2]{Cromwell} for more details. Let $\overline{\D}$ be the closed unit disk. Recall that a homeomorphism $h:S^1 \times \overline{\D} \to T \subset S^3$ is called a \emph{framing} of $T$. The two simple closed curves $h(S^1 \times *)$ and $h(*\times \partial \overline{\D})$, where $*$ denotes a point in $S^1$ or $\partial \overline{\D}$, respectively, are called a \emph{longitude} and \emph{meridian} of $T$. Up to an ambient isotopy of $T$, any two meridians of $T$ are equivalent. However, there are only two longitudes of $T$ that have the property of being nullhomologous in the complement of $T$, up to isotopy. A framing producing such a longitude is called the \emph{preferred framing} of $T$. A \emph{core curve} of $T$ is then $h(S^1 \times \{ 0 \})$, where $0$ denotes the center of the disk $\overline{\D}$. Up to isotopy of $S^3$ the core curve does not depend on $h$. In fact, the regular neighborhood of any simple closed curve $t\subset S^3$ is a solid torus uniquely determined up to isotopy. So, as far as knots and links are concerned, the behavior of a solid torus $T$ is identical to that of any core curve $t\subset T$. For convenience, we summarize the most relevant points and conclusions in the following lemma.

\begin{lemma}\label{lem:TorusKnot}
	Let $T \subset S^3$ be a solid torus, and let $t$ be any core curve of $T$. Then:
	\begin{enumerate}[(i)]
		\item $T$ is ambiently isotopic to a regular neighborhood of $t$. Hence any linking behavior of $T$ is equivalent to that of $t$.
		\item In particular, $T$ is unknotted if and only if $t$ is unknotted.
		\item The closure of $S^3\setminus T$ is an unknotted solid torus if and only if $T$ is unknotted.
	\end{enumerate}
\end{lemma}

We also require the following relationship between a solid torus $T$ and sequentially linked figure-eight curves.
\begin{lemma}\label{lem:Hopf}
	Let $\gamma_1^1, \gamma_1^2, \gamma_2^1, \gamma_2^2, \gamma_3^1,$ and $\gamma_3^2$ be unknotted topological circles in $S^3$ satisfying the following:
	\begin{itemize}
	\item For each $i$, $\gamma_i^1$ and $\gamma_i^2$ intersect at a point.
	\item The pairs $(\gamma_1^2, \gamma_2^1)$ and $(\gamma_2^2, \gamma_3^1)$ each form Hopf links.
	\item The pairs $(\gamma_1^2, \gamma_2^2)$ and $(\gamma_2^1, \gamma_3^1)$ are each unlinks.
	\end{itemize}
	
	Define $\gamma_i = \gamma_i^1\cup \gamma_i^2$ for each $i$. Suppose that $T$ is a solid torus containing $\gamma_2$. Then $T$ contains $\gamma_1$ or $\gamma_3$.
\end{lemma}
\begin{proof}
	Observe that, by construction, the $\gamma_i$ are figure-eight curves.
	
	Suppose that $T$ does not contain either $\gamma_1$ or $\gamma_3$. Then both $\gamma_{1}^2$ and $\gamma_{3}^1$ are outside of $T$. Since $\gamma_{2}^1, \gamma_{2}^2 \subset T$ are unknotted and linked with the unknotted circles $\gamma_{1}^2,\gamma_{3}^1 \subset S^3 \setminus T$, respectively, we have by Lemma \ref{lem:TorusKnot} that $T$ is itself unknotted, and that both $\gamma_{2}^1$ and $\gamma_{2}^2$ serve as core curves for $T$. By the same lemma, $\overline{S^3 \setminus T}$ is an unknotted solid torus, and both $\gamma_{1}^2$ and $\gamma_{3}^1$ are core curves for $\overline{S^3 \setminus T}$. So $\gamma_{1}^2$ and $\gamma_{3}^1$ must be homotopic in $\overline{S^3 \setminus T}$. But they are not homotopic in $\overline{S^3 \setminus \gamma_2}$, of which $\overline{S^3 \setminus T}$ is a subset. This is a contradiction.
\end{proof}
\subsection{Idea of the proof}

For convenience, the defining sequence $(X_n)_{n=0}^\infty$ for $X$ constructed in Section \ref{sec:cantorcon} will henceforth be called the \emph{standard defining sequence} for $X$. 

The main idea behind the proof is to suppose that for some $x\in X$, we have $g_x(X) = 1$. Then there must exist an alternate defining sequence $(M_i)$ of $X$ that contains solid tori containing $x$ of arbitrarily small diameter which accumulate on $x$. We will suppose that $T$ is one of these solid tori and study how it must interact with the standard defining sequence.

Since we can choose $T$ so that $\operatorname{diam} (T)$ is as small as we like, we can assume that $T$ is contained in the interior of some solid double torus $X_{I,J}$ from the standard defining sequence. Moreover, since $\partial T$ and $X$ are compact sets in $\R^3$ which do not intersect, we have $d(\partial T, X) = \delta>0$. 
We can therefore find components of the standard defining sequence whose diameters are small enough so that they cover $X\cap T$ and are contained in the interior of $T$. More precisely, we can find $N\in \N$ so that $\partial X_{I+n} \cap \partial T = \emptyset$ for all $n \geq N$.

In other words, at some level of the standard defining sequence, $\partial T$ must separate some components of $X_{I+N}$ from the other components. Without loss of generality, we will henceforth assume that $\partial T \cap \partial X_n = \emptyset$ for all $n \in \N$. The technical part of the proof is then to show that $\partial T$ cannot separate double tori from the standard defining sequence. In fact, we will show that, if $T$ contains a component of $X_n$ for some $n$, then $T$ contains all of $X_n$.

\subsection{Hopf links and filling disks}\label{sec:hlfd}
Recall the following notation from Section \ref{sec:cantorcon}:
\begin{itemize}
	\item We have square solid tori $X_0^1$ and $X_0^2$ with square core curves $\gamma_0^1$ and $\gamma_0^2$, respectively. Allowing $\gamma_0^1$ and $\gamma_0^2$ to intersect at a corner yields a solid double torus $X_0 = X_0^1 \cup X_0^2$, with core curve $\gamma_0 = \gamma_0^1 \cup \gamma_0^2$.
	\item Under the sense-preserving similarities $\phi_j$ for $j=1,\ldots, m$, we have solid double tori $X_{1,j} = \phi_j(X_0)$ with core curves $\gamma_{1,j} = \phi_j(\gamma_0)$. Each $X_{1,j}$ is also the union of overlapping solid tori $X_{1,j}^i = \phi_j(X_0^i)$ for $i = 1,2$, with corresponding core squares $\gamma_{1,j}^i$. Here the linking between subsequent double tori is such that $\gamma_{1,j}^2$ and $\gamma_{1,j+1}^1$ form a Hopf link, modulo $m$.
	\item Under iteration of the similarities, we have double tori $X_{n,j}$ for $n\in \N$ and $1\leq j \leq m^n$ which are the union of solid tori $X_{n,j}^i$ for $i = 1,2$. Similarly we have core curves $\gamma_{n,j} = \gamma_{n,j}^1 \cup \gamma_{n,j}^2$. We assume the labeling of the $X_{n,j}$ is such that given $j$, $X_{n,j}$ and $X_{n,j+1}$ are always linked, where $j$ and $j+1$ are taken modulo $m/2$.
We also have $X_n = \bigcup_{j=1}^{m^n} X_{n,j}$ for all $n$.
\end{itemize}

To this we add some new objects. Note that every $\gamma_{n,j}^i$ is a geometric square, and hence an unknotted planar topological circle in $\R^3$. Hence there exist closed topological disks in $\R^3$ bounded by the $\gamma_{n,j}^i$. Such a disk is known as a \emph{filling disk}. We isolate the unique filling disk that is a filled geometric square living in the same plane as $\gamma_{n,j}^i$. We refer to this filling disk as the \emph{canonical filling disk} of $\gamma_{n,j}^i$, and denote it by $D_{n,j}^i$.

Now let $n \in \N$, and let $T \subset \interior X_0$ be a solid torus containing some component $X_{n,j}$ of $X_n$, and such that $\partial T \cap \partial X_k = \emptyset$ for all $k$. Since the tori we are examining may be arbitrarily small, we may assume that $T$ is contained in some $X_{N,J}$ where $0< N < n$. To examine whether $\partial T$ can separate $X_{n,j}$ from other components of $X_n$, we require the following lemmas.

\begin{lemma}\label{lem:somenotall}
	Suppose that $T$ contains some, but not all, components of $X_n$ contained in the same component $X_{n-1,k}$ of $X_{n-1}$. Then $T$ is contained in $X_{n-1,k}$.
\end{lemma}
\begin{proof}
	Suppose for the sake of contradiction that $T$ is not entirely contained in $X_{n-1,k}$. Then $T$ contains points outside of $X_{n-1,k}$, and hence some points of $\partial T$ lie outside of $X_{n-1,k}$. Since $T$ contains some, but not all, of $X_{n-1,k}\cap X_n$, some points of $\partial T$ must be contained in $X_{n-1,k}$. This is a contradiction, since then $\partial T \cap \partial X_{n-1,k} \neq \emptyset$.
\end{proof}

\begin{lemma}\label{lem:ComponentHopf}
	Let $X_{n,j-1}, X_{n,j}$, and $X_{n,j+1}$ be sequentially linked components of $X_n$, and suppose that $T$ contains $X_{n,j}$. Then $T$ contains $X_{n, j-1}$ or $X_{n,j+1}$.
\end{lemma}
\begin{proof}
	Consider the core curves of the given double tori, namely $\gamma_{n,j-1}$, $\gamma_{n,j}$, and $\gamma_{n,j+1}$. Since $T$ contains $\gamma_{n,j}$, we have by Lemma \ref{lem:Hopf} that $T$ contains $\gamma_{n,j-1}$ or $\gamma_{n,j+1}$. That $T$ contains the corresponding double torus follows from the standing assumption that $\partial T \cap \partial X_k = \emptyset$ for all $k$.
\end{proof}

Repeated application of Lemma \ref{lem:ComponentHopf} yields the following.

\begin{corollary}\label{lem:loop}
Suppose that $X_{n,j}$ is a component of $X_n$ contained in the component $X_{n-1,k}$ of $X_{n-1}$, and that $\gamma_{n-1,k}^i$ is the square segment of $\gamma_{n-1,k}$ along which $X_{n,j}$ is arranged. If $T$ contains $X_{n,j}$, then $T$ contains all components of $X_n$ that are arranged along $\gamma_{n-1,k}^i$.
\end{corollary}

One way of regarding Corollary \ref{lem:loop} is to say that if $T$ contains a double torus from $X_n$, then $T$ contains all the linked double tori on the same loop. We next need the following topological lemma.

\begin{lemma}\label{lem:hopffilldisk}
Let $\Gamma_1,\Gamma_2$ be unknotted loops in $S^3$ that are linked via a Hopf link. Suppose that $\Gamma_1\cup \Gamma_2$ is contained in an embedded topological ball $B \subset S^3$. Let $D$ be a filling disk of $\Gamma_1$. Then $D\cap \Gamma_2 \neq \emptyset$ and, moreover, $\partial B \cap D$ cannot separate $\Gamma_1$ and $\Gamma_2 \cap D$ in $D$.
\end{lemma}

\begin{proof}
Suppose that $g: S^1 \to \Gamma_1$ has an extension that, by a slight abuse of notation, we still call $g : \overline{\D} \to \Gamma_1 \cup D$. As $S^3 \setminus (\Gamma_1 \cup D)$ is contractible and as $\Gamma_2$ is not contractible in $S^3 \setminus \Gamma_1$, it follows that $\Gamma_2$ has to intersect $D$.

Towards the second claim, suppose that $Y$ is a component of $D \setminus B$. If there is no such component then evidently the claim follows. Clearly $Y$ is contractible in $S^3 \setminus \Gamma_2$, from which it follows that $g^{-1}( Y)$ is contractible in $\overline{\D} \setminus g^{-1}(\Gamma_2)$. As no component of $g^{-1} ( D\setminus B)$ can separate $S^1$ from $g^{-1}(\Gamma_2)$, we conclude that $\partial B \cap D$ cannot separate $\Gamma_1$ and $\Gamma_2 \cap D$ in $D$.
\end{proof}

We now apply Lemma \ref{lem:hopffilldisk} to our situation to show that neighboring components of $X_n$ that are contained in $T$ can be joined by a path in $T$ that stays close to the components.

\begin{lemma}
\label{lem:filldisk}
Suppose that the torus $T$ contains components $X_{n,j}$ and $X_{n,j+1}$ of $X_n$ which are linked via the squares $\gamma_{n,j}^2$ and $\gamma_{n,j+1}^1$. Let $B$ be a topological ball containing $\gamma_{n,j}^2$ and $\gamma_{n,j+1}^1$ and not intersecting any other component of $X_n$. Then there exists a path $p$ joining $X_{n,j}$ and $X_{n,j+1}$ inside $T\cap B$.
\end{lemma}

\begin{proof}
If the canonical filling disk $D_{n,j}^2$ contains a continuum connecting $\gamma_{n,j}^2$ and $\gamma_{n,j+1}^1$ then we are done. Otherwise, there must be a topological circle $a \subset D_{n,j}^2 \cap (S^3 \setminus T)$ separating $\gamma_{n,j}^2$ from $\gamma_{n,j+1}^1 \cap D_{n,j}^2$ in $D_{n,j}^2$.

As $\gamma_{n,j+1}^1 \subset T$ is linked with $a \subset S^3 \setminus T$, it follows that $\gamma_{n,j+1}^1$ is a core curve for $T$. By Lemma \ref{lem:TorusKnot}, this means that $T$ is an unknotted torus in $S^3$ and thus the closure of its complement $\overline{ S^3   \setminus T}$ is also an unknotted torus in $S^3$. As $a\subset \overline{S^3 \setminus T}$ is linked with $\gamma_{n,j+1}^1 \subset T$, we have that $a$ is a core curve for $\overline{ S^3 \setminus T}$. Then as $a$ is not linked with $\gamma_{n,j}^2$, we conclude that $\gamma_{n,j}^2$ is contractible in $T$.

Any homotopy relative to $\overline{S^3 \setminus T}$ which collapses $\gamma_{n,j}^2$ to a point of $\gamma_{n,j+1}^1$ sweeps out a filling disk of $\gamma_{n,j}^2$ contained in $T$. Applying Lemma \ref{lem:hopffilldisk} to this filling disk and the ball $B$ yields a path $p$ joining $\gamma_{n,j}^2$ and $\gamma_{n,j+1}^1$ in $B\cap T$.
\end{proof}

We would like the path from the previous lemma to not stray too far from a canonical filling disk. Fortunately, we can choose paths arbitrarily close to canonical filling disks. Recall that $d$ denotes the Euclidean distance on $\R^n$ and that if $x\in \R^n$ and $\Omega \subset \R^n$ then 
\[d(x,\Omega) = \inf_{y\in \Omega} d(x,y) .\]

\begin{corollary}\label{lem:distance}
Let all the assumptions from Lemma \ref{lem:filldisk} be given and let $\epsilon >0$. Then there exists a path $p_\epsilon$ joining $X_{n,j}$ and $X_{n,j+1}$ inside $T$ satisfying
	\[\sup\{d(x, E_{n,j}) \,|\, x\in p_\epsilon\} < \epsilon, \]
where $E_{n,j} = D_{n,j}^2 \cup D_{n,j+1}^1$ is the union of the canonical filling disks of $\gamma_{n,j}^2$ and $\gamma_{n,j+1}^1$.
\end{corollary}
\begin{proof}
Observe that an $\epsilon/2$-neighborhood of $E_{n,j}$ yields an embedded topological ball $B$ satisfying
	\[\sup\{d(x, E_{n,j}) \,|\, x\in B\} < \epsilon. \]
Applying Lemma \ref{lem:filldisk} to $B$ yields the path $p_\epsilon$.
\end{proof}

Considering the union of these paths with the tori in a loop of double tori now yields a genuine loop in $T$ that we can work with.

\begin{lemma}\label{lem:circle}
	Suppose that $T$ contains all components from $X_n$ that are arranged along the same square $\gamma_{n-1,k}^i$, say $\{ X_{n,j} : j \in J \}$ where $J\subset \{1,\ldots, m^n \}$ and the cardinality of $J$ is $m/2$. Then $T$ contains a topological circle $C$ that is homotopic to $\gamma_{n-1,k}^i$ via a homotopy whose image is inside the double torus $X_{n-1,k}$.
\end{lemma}

\begin{proof}
	Let $\epsilon = \text{diam}(X_{n,j})/1000$. From Corollary \ref{lem:distance}, for each $j\in J$ and every linked pair of squares $(\gamma_{n,j}^2, \gamma_{n,j+1}^1)$ contained in $T$, there exists a path $p_j \subset T$ connecting $\gamma_{n,j}^2$ and $\gamma_{n,j+1}^1$ which cannot intersect any other component of $X_n$, and satisfying 
	 \[\sup\{d(x, E_{n,j}) \,|\, x\in p_j\} < \epsilon.\] Taking the union of these paths $p_j$ together with certain segments of the cores $\gamma_{n,j}$ yields a topological circle $C$.
	
	Now take an $\epsilon$-neighborhood $F_{n,j}$ of each $E_{n,j}$, and let $Y = \bigcup_{j\in J} F_{n,j}$. Observe, by construction of the standard defining sequence, that $Y$ is then a solid torus contained in $X_{n-1,k}$, with core curve $\gamma_{n-1,k}^i$. Hence $Y$ deformation retracts onto $\gamma_{n-1,k}^i$. Since $C\subset Y$ is by construction homotopically nontrivial in $Y$, it follows that $C$ is homotopic to $\gamma_{n-1,k}^i$ via a homotopy whose image is inside $X_{n-1,k}$.

\end{proof}

\subsection{Separating double tori}\label{sec:septori}

Finally we may examine whether $\partial T$ can separate components of $X_n$ from each other.
\begin{lemma}\label{lem:continuity}
	If $T$ contains $X_{n,j}$, a component of $X_n$, then $T$ contains $X_{n-1,k}\cap X_n$, where $X_{n-1,k}$ is the component of $X_{n-1}$ containing $X_{n,j}$.
\end{lemma}
\begin{proof}
	Note that $X_{n,j}$ is arranged along one of the square segments of $\gamma_{n-1,k}$, call this segment $\gamma_{n-1,k}^i$. By Corollary \ref{lem:loop}, we have that $T$ contains the loop of components of $X_n$ arranged along $\gamma_{n-1,k}^i$. Denote this loop by $A$. It remains to show that $T$ contains the complementary loop, call it $B$
	
	Suppose for the sake of contradiction that $T$ does not contain $B$. By Corollary \ref{lem:loop}, this means that $T$ does not contain any component of $B$. By Lemma \ref{lem:somenotall}, we then have that $T \subset \interior X_{n-1,k}$.
	
	By Lemma \ref{lem:circle}, since $T$ contains $A$, we have that $T$ also contains a topological circle $C$ that is homotopic to $\gamma_{n-1,k}^i$ via a homotopy whose image lies inside $X_{n-1, k}$. Since $T \subset \interior X_{n-1,k}$, the square $\gamma_{n-1,l}^{3-i}$, which is linked with $C$, is not contained in $T$. We also have, since $T$ contains $A$, that $T$ contains one of the squares involved in the four-way linking at the center of $X_{n-1,k}\cap X_n$, call it $\gamma_{n,r}^s$. Since $T$ does not contain $B$, there must also be a different square from the four-way linking that is not contained in $T$. Call this square $\gamma_{n,t}^u$, and note that it is linked with $\gamma_{n,r}^s$. 
	We summarize for clarity:
	\begin{itemize}
		\item $T$ contains the topological circle $C$.
		\item $T$ does not contain the square $\gamma_{n-1,l}^{3-i}$, which is linked with $C$.
		\item $T$ contains the square $\gamma_{n,r}^s$.
		\item $T$ does not contain the square $\gamma_{n,t}^u$, which is linked with $\gamma_{n,r}^s$.
	\end{itemize}

	By a similar argument as in the proof of Lemma \ref{lem:Hopf}, this requires that $\gamma_{n-1,l}^{3-i}$ and $\gamma_{n,t}^u$ be homotopic in $S^3 \setminus T$. But by construction of the standard defining sequence, since these two squares arise from different levels of the sequence, they are not homotopic in $S^3 \setminus (X_n \cap T)$, of which $S^3 \setminus T$ is a subset. This is a contradiction.
	
	Hence $T$ contains $B$, and so $T$ contains $A\cup B = X_{n-1,k}\cap X_n$.
\end{proof}

\begin{lemma}\label{lem:containsXn}
	If $T$ contains $X_{n,j}$, a component of $X_n$, then $T$ contains all of $X_n$.
\end{lemma}
\begin{proof}
	By Lemma \ref{lem:continuity}, since $T$ contains $X_{n,j}$, we have that $T$ contains $X_{n-1,k}\cap X_n$, where $X_{n-1,k}$ is the component of $X_{n-1}$ containing $X_{n,j}$. By Lemma \ref{lem:circle}, $T$ then contains topological circles $C_{n-1,k}^1$ and $C_{n-1,k}^2$, which are homotopic to the squares $\gamma_{n-1,k}^1$ and $\gamma_{n-1,k}^2$, respectively, inside $X_{n-1,k}$. 
	 By Corollary \ref{lem:distance}, $T$ also contains a path $p$ connecting $C_{n-1,k}^1$ and $C_{n-1,k}^2$ which remains arbitrarily close to a canonical filling disk at the four-way linking in $X_{n-1,k}\cap X_n$. To sum up,
		\[C_{n-1,k} := C_{n-1,k}^1 \cup p \cup C_{n-1,k}^2 \subset T, \]
	where, by construction, $C_{n-1,k}$ is homotopic to the core curve $\gamma_{n-1,k}$ inside $X_{n-1,k}$, and therefore has the same linking behavior as $\gamma_{n-1,k}$.
	
	Note that $C_{n-1,k}$ is a figure-eight curve consisting of unknotted loops, and that $C_{n-1,k}$ is sequentially linked with the core curves $\gamma_{n-1,k-1}$ and $\gamma_{n-1,k+1}$. By Lemma \ref{lem:Hopf}, $T$ contains at least one of these core curves. In fact, the argument of Lemma \ref{lem:ComponentHopf} shows that $T$ actually contains at least one linking neighbor of $X_{n-1,k}$. Following the same line of reasoning as the previous several lemmas, we can conclude in a similar manner to the proof of Lemma \ref{lem:continuity} that $T$ contains every other component from $X_{n-1}$ contained in the same component of $X_{n-2}$ as $X_{n-1,k}$. Call this component $X_{n-2,l}$.
	
	We may now construct a figure-eight curve $C_{n-2,l}$ to which $\gamma_{n-2,l}$ is homotopic inside $X_{n-2,l}$, and such that $C_{n-2,l}$ has the same linking behavior as $\gamma_{n-2,l}$. By the same reasoning as in the previous paragraph, we conclude that $T$ contains every other component from $X_{n-2}$ contained in the same component of $X_{n-3}$.
	
	Repeating this process finitely many times, we obtain the following:
		\begin{itemize}
		\item $T$ contains all but one component of $X_1$.
		\item Out of the components of $X_2$ contained in the missed component of $X_1$, $T$ contains all but one.
		\item Out of the components of $X_3$ contained in the missed component of $X_2$, $T$ contains all but one.
		\item This carries down to the missed component $X_{n-1,k}$ of $X_{n-1}$, but $T$ contains all of the components of $X_{n} \cap X_{n-1,k}$. 
		\end{itemize}
	
	Since $X_n \subset X_i$ for all $i \leq n$, we conclude that $T$ contains all of $X_n$.
\end{proof}

\begin{proof}[Proof of Theorem \ref{thm:cantorset}]
	Let $X$ be the Cantor set constructed in Section \ref{sec:cansetconstruction}, and suppose for the sake of contradiction that there exists $x\in X$ such that $g_x(X) =1$. Then there exists a defining sequence $(M_i)$ of $X$ having components that are solid tori, a sequence of which accumulates to $x$. Let $T$ be such a small toroidal neighborhood of $x$.
	
	At some level $n$ of the standard defining sequence, we must then have that $T$ contains a component of $X_n$. By Lemma \ref{lem:containsXn}, $T$ contains all of $X_n$. However, we have $\text{diam}(T) \geq \text{diam}(\gamma_0)$, where $\gamma_0$ is the core curve of the double torus $X_0$. This means that $(M_i)$ cannot contain arbitrarily small solid tori accumulating to $x$, which is a contradiction.
	
	We conclude that $g_{x}(X) \geq 2$ for all $x \in X$, and so in particular $g(X) \geq 2$. Since the standard defining sequence consists entirely of double tori, $g(X) \leq 2$, and hence $g(X) = 2$.
\end{proof}

\end{document}